\documentclass[11pt]{article}

\usepackage[english]{babel}
\usepackage{upref,amsfonts,amsxtra,latexsym,color}
\usepackage{amssymb,amsmath,amsthm,epsf,mathrsfs,verbatim,hyperref}
\usepackage[all,color, matrix, arrow]{xy}
\usepackage[left=2.9cm,top=2.5cm,right=2.9cm,bottom=2cm,bindingoffset=0.5cm]{geometry}

\usepackage{tikz}
 
\newtheorem{theorem}{Theorem}[subsection]
\newtheorem{lemma}[theorem]{Lemma}
\newtheorem{corollary}[theorem]{Corollary}

\theoremstyle{definition}

\numberwithin{equation}{section}
\numberwithin{theorem}{section}

\newcommand{\ep}{{\varepsilon}}

\newcommand{\C}{{\mathbb C}}

\newcommand{\Cs}{{$C^*$-al\-ge\-bra}}

\DeclareMathOperator{\cconv}{\overline{conv}}
\DeclareMathOperator{\re}{Re}
\DeclareMathOperator{\ra}{Rad}

\DeclareMathOperator{\wscconv}{\overline{conv}^{\mathit{w}\kern-.05em\raise.1ex\hbox{$\scriptscriptstyle *$}}\kern-.2em}
\DeclareMathOperator{\wcconv}{\overline{conv}^{\mathit{w}}\kern-.2em}

\date{}
\begin{document} 

\title{A new look at $C^*$-simplicity \\and the u\-ni\-que trace property of a group}

\author{Uffe Haagerup\footnote{Uffe Haagerup tragically passed away on July 5, 2015. The results in this article were proven by Haagerup in the early Spring of 2015, and privately communicated to his (and Magdalena Musat's) PhD student Kristian Knudsen Olesen in May 2015. Based on this, Kristian Knudsen Olesen, Magdalena Musat and Mikael R\o rdam (University of Copenhagen, Denmark) have written up this paper. }}

\maketitle

\begin{abstract} We characterize when the reduced \Cs{} of a non-trivial group has unique tracial state, respectively, is simple, in terms of Dixmier-type properties of the group $C^*$-algebra. We also give a simple proof of the recent result by Breuillard, Kalantar, Kennedy and Ozawa that the reduced \Cs{} of a group has unique tracial state if and only if the amenable radical of the group is trivial.
\end{abstract}

\section{Introduction}

\noindent  It was shown by Murray and von Neumann that the von Neumann algebra $\mathcal{L}(G)$ of a group $G$ is a factor (necessarily of type II$_1$) if and only if the group $G$ is ICC (all non-trivial conjugacy classes are infinite). The analogous problem for the reduced group \Cs{} $C^*_\lambda(G)$ of the group $G$, namely characterizing when $C^*_\lambda(G)$ is simple, respectively, when does it have a unique tracial state, turned out to be far more subtle. Powers, \cite{Pow75}, proved in 1975 that the reduced group \Cs{} of the free groups (or order $\ge 2$) is simple and has unique tracial state. His result has since then been vastly generalized. Prompted by the observation that (a) and (b) below separately imply condition (c), see, for example \cite{Harpe:BLMS}, the question arose if the following three conditions for a  group $G$ are equivalent:
\begin{itemize}
\item[(a)] $C^*_\lambda(G)$ is simple (i.e., $G$ is $C^*$-simple),
\item[(b)] $C^*_\lambda(G)$ has unique tracial state (i.e., $G$ has the unique trace property),
\item[(c)] the amenable radical of $G$ is trivial.
\end{itemize}
Kalantar and Kennedy proved in \cite{KalKen:boundary} that $C^*$-simplicity of a group is equivalent to the group having a (topologically) free boundary action. In 2014, Breuillard, Kalantar, Kennedy and Ozawa, \cite{BKKO}, used this to prove that (b) and (c) are equivalent, and hence that (a) implies (b). The picture of the interconnections between the three properties above was finally completed very recently by Le Boudec, \cite{Boudec:C*-simple}, who gave examples of groups that have the unique trace property, but  are not $C^*$-simple. Hence (b) does not imply (a). Simple \Cs s in general need not have unique tracial state. For example, any metrizable Choquet simplex arises as the trace simplex of a unital simple AF-algebra, \cite{Eff:AF}. 

In this article we give new characterizations of when a group has the unique trace property and when it is $C^*$-simple in terms of (intrinsic) Dixmier-type properties of the group \Cs. We also give a more direct proof of the theorem by Breuillard, Kalantar, Kennedy and Ozawa that the unique trace property is equivalent to triviality of the amenable radical of the group. This proof uses Furman's characterization (cf.\ \cite{Furman-2003}) of the amenable radical of a group as consisting of those elements that act trivially under any boundary action. 

In the very recent preprint \cite{Kennedy:C*-simple}, Kennedy has independently obtained results similar to those in Theorem~\ref{smpl} of this paper, characterizing when a group is $C^*$-simple.

\section{Boundary actions} \label{sec:2}

\noindent Recall that an action of a group\footnote{Throughout the paper, groups are assumed to be discrete.} $G$ on a compact Hausdorff space $X$ is said to be \emph{strongly proximal} if for each probability measure $\mu$ on $X$, the weak$^*$-closure of the orbit $G.\mu$ contains a point-mass $\delta_x$, for some $x \in X$. An action $G \curvearrowright X$ is said to be a \emph{boundary action} if it is strongly proximal and minimal. If this is the case, then for each $x \in X$ and  each probabiliby measure $\mu$ on $X$, there is a net $(s_i)$ in $G$ such that $s_i.\mu$ converges to $\delta_x$ in the weak$^*$-topology. Recall also that the \emph{amenable radical} of $G$ is defined to be the largest normal amenable subgroup of $G$, and it is denoted by $\ra(G)$. 

Furman, \cite{Furman-2003}, proved in 2003 the following result about (the existence of) boundary actions of a group:

\begin{theorem}[Furman] \label{thm:Furman}
Let $G$ be a group and let $t \in G$. Then $t \notin \ra(G)$ if and only if there is a boundary action of $G$ on some compact Hausdorff space $X$ such that $t$ acts non-trivially on $X$.
\end{theorem}

\noindent We denote by $\lambda$ the left-regular representation of the group $G$ on $\ell^2(G)$, and by $C^*_\lambda(G)$ the (associated) reduced group \Cs. A group $G$ is said to be \emph{$C^*$-simple} if $C^*_\lambda(G)$ is a simple \Cs, and it is said to have the \emph{unique trace property} if the canonical trace on $C^*_\lambda(G)$, here denoted by $\tau_0$, is the only tracial state on $C^*_\lambda(G)$.

Kalantar and Kennedy proved in  \cite{KalKen:boundary} that a group $G$ is $C^*$-simple if and only if it has a topologically free boundary action on some compact Hausdorff space. It was observed in \cite[Proposition 2.5]{BKKO} that the action of $G$ on its universal boundary $\partial_F G$ of $G$ is free if it is topologically free, and hence that the following theorem holds:

\begin{theorem}[Breuillard--Kalantar--Kennedy--Ozawa] \label{thm:KK}
Let $G$ be a group. Then $C^*_\lambda(G)$ is simple if and only if there is a free boundary action of $G$ on some compact Hausdorff space.
\end{theorem}

\noindent

\section{Groups with the unique trace property}

\noindent In this section we give a new and elementary proof of one of the main theorems from \cite{BKKO}, namely that a group has the unique trace property if and only if it has trivial amenable radical. The proof uses Theorem~\ref{thm:Furman} by Furman, quoted above. 

When a group $G$ acts on a compact Hausdorff space $X$, we can then form the reduced crossed product \Cs{} $C(X) \rtimes_r G$, see \cite[Chapter 4]{BrownOzawa}, which in a natural way contains both $C^*_\lambda(G)$ and $C(X)$. These two subalgebras are related as follows: $\lambda(t) f \lambda(t)^* = t.f$, for all $t \in G$ and $f \in C(X)$, where $(t.f)(x) = f(t^{-1}.x)$, for $x \in X$. 

\begin{lemma} \label{lm:multiplicative}
Let $G$ be a group acting on a compact Hausdorff space $X$, let $x \in X$, and let $\varphi$ be a state on $C(X) \rtimes_r G$ whose restriction to $C(X)$ is the point-evaluation $\delta_x$. Then $\varphi(\lambda(t)) = 0$, for each $t \in G$ for which $t.x \ne x$. 
\end{lemma}

\begin{proof} The assumptions in the lemma ensure that $C(X)$ is contained in the multiplicative domain of $\varphi$, see \cite[Proposition 1.5.7]{BrownOzawa}, so 
$$\varphi(\lambda(t)) f(x) = \varphi(\lambda(t) f) = \varphi((t.f) \lambda(t)) = f(t^{-1}.x) \varphi(\lambda(t)),$$
for each $f \in C(X)$ and each $t \in G$. This clearly entails that $\varphi(\lambda(t)) = 0$ when $t^{-1}.x \ne x$, which again happens precisely when $t.x \ne x$. 
\end{proof}

\noindent If $\varphi$ is a state on $C^*_\lambda(G)$, or on the crossed product $C(X) \rtimes_r G$, and if $t \in G$,  let $t.\varphi$ denote the state given by $(t.\varphi)(a) = \varphi(\lambda(t)a\lambda(t)^*)$, where $a$ belongs to $C^*_\lambda(G)$, respectively, to $C(X) \rtimes_r G$.

\begin{lemma} \label{lm:extension}
Let $\tau$ be a tracial state on $C^*_\lambda(G)$, let $G \curvearrowright X$ be a boundary action, and let $x \in X$. Then $\tau$ extends to a state on $C(X) \rtimes_r G$ whose restriction to $C(X)$ is point-evaluation $\delta_x$.
\end{lemma}

\begin{proof} Extend $\tau$ to any state $\psi$ on $C(X) \rtimes_r G$, and let $\rho$ be the restriction of $\psi$ to $C(X)$. By the assumption that $G \curvearrowright X$ is a boundary action, there is a net $(s_i)$ in $G$ such that $s_i.\rho$ converges to $\delta_x$ in the weak$^*$-topology. By possibly passing to a subnet we may assume that $s_i.\psi$ converges to some state $\varphi$ on $C(X) \rtimes_r G$. The restriction of $\varphi$ to $C(X)$ is then equal to $\delta_x$. Moreover, since for all $s,t \in G$,
$$(s.\psi)(\lambda(t)) = \psi(\lambda(sts^{-1})) = \tau(\lambda(sts^{-1})) = \tau(\lambda(t)),$$
we see that $\varphi(\lambda(t)) = \tau(\lambda(t))$, for all $t \in G$. This shows that $\varphi$ and $\tau$ agree on $C^*_\lambda(G)$. 
\end{proof}

\noindent  We can now give a new and simpler proof of one of the main theorems from \cite{BKKO}:

\begin{theorem}[Breuillard--Kalantar--Kennedy--Ozawa] \label{thm:BKKO}
Let $G$ be a group and let $t \in G$. Then $\tau(\lambda(t)) = 0$, for every tracial state $\tau$ on $C^*_\lambda(G)$, if and only if $t \notin \ra(G)$. 
In particular, $C^*_\lambda(G)$ has a unique tracial state if and only if $\ra(G)$ is trivial. 
\end{theorem}

\begin{proof} Suppose first that $t \notin \ra(G)$. Then by Theorem~\ref{thm:Furman}  (Furman), there is a boundary action $G \curvearrowright X$ such that $t.x \ne x$, for some $x \in X$. Let $\tau$ be a tracial state on $C^*_\lambda(G)$. By Lemma~\ref{lm:extension}, there is a state $\varphi$ on $C(X) \rtimes_r G$ which extends $\tau$ and whose restriction to $C(X)$ is point-evaluation at $x$. By Lemma~\ref{lm:multiplicative}, it follows that $\tau(\lambda(t)) = \varphi(\lambda(t)) = 0$.

The ``only if'' part follows from the well-known fact, see, for example, Proposition 3 in \cite{Harpe:BLMS} (and its proof therein), that whenever $N$ is a normal amenable subgroup of $G$, then the canonical homomorphism $\C G \to \C (G/N)$ extends to a $^*$-homomorphism $C_\lambda^*(G) \to C_\lambda^*(G/N)$. Using this fact with $N=\ra(G)$, and composing the resulting $^*$-homomorphism   $C_\lambda^*(G) \to C_\lambda^*(G/\ra(G))$ with the canonical trace on $C^*_\lambda(G/\ra(G))$, we obtain a tracial state $\tau$ on $C^*_\lambda(G)$ which satisfies $\tau(\lambda(t)) = 1$, for all $t \in \ra(G)$. 

The last claim of the theorem follows from the fact that the canonical trace on $C^*_\lambda(G)$ is the unique tracial state which vanishes on $\lambda(t)$, for all $t \ne e$. 
\end{proof}

\section{$C^*$-simplicity and the u\-ni\-que trace property of groups}

\noindent This section contains our main results that provide new characterizations of the unique trace property and $C^*$-simplicity of a group in terms of Dixmier-type properties of the group \Cs.

The lemma below is well-known, see for example \cite[Lemma 2.1(c)]{Ol-Osin}. We include a proof for completeness.

\begin{lemma} \label{lm:cone}
Let $x, y \in C^*_\lambda(G)$ be finite positive linear combinations of elements from $\{\lambda(t) \mid t \in G\}$. Then $\|x+y\| \ge \|x\|$.
\end{lemma}

\begin{proof} Let $\ell^2(G)_+$ denote the ``positive cone'' consisting of all vectors $\xi \in \ell^2(G)$ for which $\langle \xi, e_t \rangle \ge 0$, for all $t \in G$. Here $(e_t)_{t \in G}$ denotes the standard orthonormal basis of $\ell^2(G)$. It is clear that $\langle \xi, \eta \rangle \ge 0$, for all $\xi, \eta \in \ell^2(G)_+$. Moreover, each element $z \in C^*_\lambda(G)$ which is a finite positive linear combination of elements from $\{\lambda(t) \mid t \in G\}$ maps $\ell^2(G)_+$ into $\ell^2(G)_+$, and satisfies $|\langle z \xi, \eta \rangle | \le \langle z |\xi|, |\eta| \rangle$, from which it follows that
\begin{equation*} 
\|z\| = \sup \{ \langle z \xi, \eta \rangle \mid \xi, \eta \in \ell^2(G)_+, \;  \|\xi\| \le 1, \; \|\eta\| \le 1\}.
\end{equation*}
The conclusion follows now easily.
\end{proof}

\begin{lemma} \label{sfc}
    Let $G$ be a group and let $t \in G$.  The following conditions are equivalent:
\begin{enumerate}
\item $0 \notin \cconv\{\lambda(sts^{-1}) \mid s \in G \}$, 
\item $0 \notin \cconv\big\{\lambda(sts^{-1}) + \lambda(sts^{-1})^* \mid s \in G \big\}$,
\item there exist a self-adjoint linear functional~$\omega$ on $C^*_\lambda(G)$ of norm~1 and a constant $c>0$ such that $\re \omega(\lambda(sts^{-1})) \geq c$, for all $s\in G$.
\end{enumerate}
\end{lemma}

\begin{proof} For all $c_1, \ldots, c_n \geq 0$ with $\sum_{k=1}^n c_k = 1$, and all $s_1,
    \ldots, s_n \in G$ we have
        \begin{equation*}
        \Bigl\| \sum_{k=1}^n c_k \lambda(s_k t s_k^{-1}) \Bigr\|
            \leq \Bigl\| \sum_{k=1}^n c_k \big(\lambda(s_k t s_k^{-1}) +
                \lambda(s_k t s_k^{-1})^*\big) \Bigr\|
            \leq 2\Bigl\| \sum_{k=1}^n c_k \lambda(s_k t s_k^{-1})
                \Bigr\|.
        \end{equation*}
The first inequality holds by Lemma~\ref{lm:cone}, and the second by the triangle inequality. Together, the two inequalities show that (i) and (ii) are equivalent. 

The fact that (ii) implies (iii) follows from a standard application of the Hahn--Banach separation theorem on the real vector space of  self-adjoint elements in $C^*_\lambda(G)$, while it is clear that (iii) implies (i).
\end{proof}

\noindent The theorem below sharpens the first part of Theorem~\ref{thm:BKKO}.

\begin{theorem} \label{thm:2}
    Let $G$ be a group and let $t \in G$. Then $t \notin \ra(G)$ if and only if
\begin{equation} \tag{$\dagger$}
0 \in \cconv\{ \lambda(sts^{-1}) \mid s \in G \}.
\end{equation}
\end{theorem}

\begin{proof}  If ($\dagger$) holds, then $\tau(\lambda(t)) = 0$, for every tracial state $\tau$ on $C^*_\lambda(G)$, whence $t \notin \ra(G)$ by Theorem~\ref{thm:BKKO}.

For the converse implication, suppose that ($\dagger$) does not hold, and assume to reach a
    contradiction that $t \notin \ra(G)$. By Theorem~\ref{thm:Furman} (Furman), there is a boundary action $G \curvearrowright X$ and some $x \in X$ such that $t.x \ne x$. 

    By Lemma~\ref{sfc}, there exist a self-adjoint linear functional~$\omega$
    on $C^*_\lambda(G)$ of norm~1 and a constant~$c>0$ such that $\re
    \omega(\lambda(sts^{-1})) \geq c$, for all $s\in G$.  Let $\omega = \omega_+
    - \omega_-$  be the Jordan decomposition of $\omega$, where
   $\omega_+$ and~$\omega_-$ are positive linear functionals with $\|\omega\| = \|\omega_+\| + \| \omega_-\|$.  Observe that $\omega_+ + \omega_-$ is a state, because $\|\omega\|=1$. 

Further, extend $\omega_\pm$ to positive linear functionals $\psi_\pm$ on $C(X) \rtimes_r G$ with $\|\psi_\pm\| = \|\omega_\pm\|$. Then $\psi_+ + \psi_-$ is a state on $C(X) \rtimes_r G$ which extends the state $\omega_+ + \omega_-$, and, moreover,  $\psi_+ - \psi_-$ is a self-adjoint linear functional which extends $\omega$. 

Let $\rho$ be the restriction of $\psi_++\psi_-$ to $C(X)$. As in the proof of Lemma~\ref{lm:extension}, since $G \curvearrowright X$ is a boundary action, there is a net $(s_i)$ in $G$ such that $s_i.\rho$ converges to the point-mass $\delta_x$ in the weak$^*$-topology. Upon possibly passing to a subnet, we can assume that $s_i.\psi_\pm$ converge in the weak$^*$-topology to positive functionals $\varphi_\pm$ on $C(X) \rtimes_r G$ (necessarily with the same norms as  $\psi_\pm$). The restriction of $\varphi_+ + \varphi_-$ to $C(X)$ is equal to $\delta_x$, which is a pure state on $C(X)$, so the restriction of $\varphi_\pm$ to $C(X)$ must be equal to $\|\varphi_\pm\| \! \cdot \! \delta_x$. We can now use Lemma~\ref{lm:multiplicative} (applied to suitable multiples of the positive linear functionals $\varphi_\pm$) to conclude that $\varphi_\pm(\lambda(t)) = 0$. Hence
\begin{eqnarray*}
0 & = & \varphi_+(\lambda(t)) -\varphi_-(\lambda(t)) \; = \; {\textstyle{\lim_{i}}} \big(s_i.\psi_+(\lambda(t)) - s_i.\psi_-(\lambda(t))\big) \\ &=& {\textstyle{\lim_{i}}} \big(s_i.\omega_+(\lambda(t)) - s_i.\omega_-(\lambda(t))\big)  \; =\;  {\textstyle{\lim_{i}}} \, s_i.\omega(\lambda(t)) \; = \; {\textstyle{\lim_i}} \, \omega(\lambda(s_its_i^{-1})),
\end{eqnarray*}
which  contradicts the fact that $\re   \omega(\lambda(sts^{-1})) \geq c > 0$, for all $s\in G$. 
\end{proof}

\noindent It is well-known that groups with trivial amenable radical are ICC. This fact also follows from Theorem~\ref{thm:2},  since  ($\dagger$) can only hold for elements $t \in G$ belonging to an infinite conjugacy class.

From Theorem~\ref{thm:2} and Theorem~\ref{thm:BKKO} we obtain the following:

\begin{corollary} \label{cor:unique_trace-2}
    Let $G$ be a group.  Then $C^*_\lambda(G)$ has a unique
    tracial state if and only if
    \begin{equation*}
        0 \in \cconv\{ \lambda(sts^{-1}) \mid s \in G \},
        \end{equation*}
    for all $t \in G \setminus \{e\}$.
    \end{corollary}

\noindent  Using Theorem~\ref{thm:KK} (Breuillard--Kalantar--Kennedy--Ozawa), we can characterize $C^*$-simple groups as follows:

\begin{theorem}  \label{smpl}
Let $G$ be a group and let $\tau_0$ denote the canonical tracial state on $C^*_\lambda(G)$. Then the following are equivalent:
    \begin{enumerate}
\item $C^*_\lambda(G)$ is simple,
\item $\tau_0 \in \overline{\{s.\varphi \mid s \in G\}}^{\, w^*}$, for each state $\varphi$ on $C^*_\lambda(G)$,
\item $\tau_0 \in \wscconv\{s.\varphi \mid s \in G\}$,  for each state $\varphi$ on $C^*_\lambda(G)$,
\item $\omega(1) \! \cdot \! \tau_0 \in \wscconv\{s.\omega \mid s \in G\}$,  for each bounded linear func\-tional $\omega$ on $C^*_\lambda(G)$,
\item for all $t_1, t_2,\ldots, t_m \in G \setminus\{e\}$,
\begin{equation*}  0 \in \cconv\big\{\lambda(s)\big(\lambda(t_1) + \lambda(t_2)+\cdots + \lambda(t_m)\big)\lambda(s)^* \mid s \in G\big\},
\end{equation*}
\item for all $t_1, t_2, \ldots, t_m$ in $G \setminus \{e\}$ and all $\ep > 0$, there exist  $s_1, s_2,\ldots, s_n \in  G$ such that
\begin{equation*}
\Bigl\| \sum_{k=1}^n \frac{1}{n} \lambda(s_k t_j s_k^{-1})\Bigr\| < \ep,
\end{equation*}
for $j=1,2, \dots, m$.
\end{enumerate}
\end{theorem}

\begin{proof} (i) $\Rightarrow$ (ii). Let $\varphi$ be a state on $C^*_\lambda(G)$. By Theorem~\ref{thm:KK} (Breuillard--Kalantar--Kennedy--Ozawa) there is a free boundary action $G \curvearrowright X$. Take any $x \in X$.  Extend $\varphi$ to a state $\psi$ on $C(X) \rtimes_r G$ and let $\rho$ be the restriction of $\psi$ to $C(X)$. Since $G \curvearrowright X$ is a boundary action, there is a net $(s_i)$ in $G$ such that $s_i.\rho$ converges to the point-evaluation $\delta_x$ in the weak$^*$-topology. Upon possibly passing to a subnet we may assume that $s_i.\psi$ converges to some state $\psi'$ on $C(X) \rtimes G$. Note that $s_i.\varphi$ converges to the restriction of $\psi'$ to $C^*_\lambda(G)$. 

The restriction of $\psi'$ to $C(X)$ is $\delta_x$, so by Lemma~\ref{lm:multiplicative}, together with the fact that the action of $G$ on $X$ is free, we deduce that $\psi'(\lambda(t)) = 0$, for all $t \in G\setminus\{e\}$. The restriction of $\psi'$ to $C^*_\lambda(G)$ is therefore equal to  $\tau_0$. We conclude that $s_i.\varphi$ converges to $\tau_0$.

It is trivial that (ii) implies (iii).

(iii) $\Rightarrow$ (iv). Fix states $\varphi_1, \varphi_2, \dots, \varphi_m$  on $C^*_\lambda(G)$. The set
$$\wscconv\{(s.\varphi_1,s.\varphi_2, \dots, s.\varphi_m) \mid s \in G\}$$
is a weak$^*$ closed $G$-invariant convex subset of the set of $m$-tuples of the set of states on $C^*_\lambda(G)$. Repeated applications of (iii) show that the $m$-tuple $(\tau_0, \tau_0, \dots, \tau_0)$ belongs to this set. It follows that for each finite subset $F$ of $C^*_\lambda(G)$, and $\ep >0$, there exist $s_1, s_2, \dots, s_n \in G$ such that 
$$\Big| \frac{1}{n} \sum_{k=1}^n s_k.\varphi_j(a) - \tau_0(a)\Big| < \ep,$$
for all $a \in F$ and all $j=1,2, \dots, m$. Since each bounded linear functional is a linear combination of finitely many (in fact, four) states, we see that (iv) holds.

(iv) $\Rightarrow$ (v). Suppose that (v) does not hold, and let $t_1, t_2, \dots, t_m$ be elements in $G \setminus \{e\}$ that witness the failure of (v). Using Lemma~\ref{lm:cone} and arguing as in the proof of Lemma~\ref{sfc}, we conclude that
$$ 0 \notin \cconv\Big\{\lambda(s)\big(\sum_{j=1}^m (\lambda(t_j) + \lambda(t_j^{-1}))\big)\lambda(s)^* \mid s \in G\Big\}.$$

By the Hahn--Banach separation theorem we obtain a self-adjoint linear functional $\omega$ on $C^*_\lambda(G)$ (of norm 1)
    and $c>0$ such that for all $s \in G$,
 \begin{equation*} 
        2 \re \omega\Bigl( \sum_{j=1}^n \lambda(s t_j s^{-1})\Bigr)
            = \omega\Bigl( \sum_{j=1}^n \lambda(s t_j s^{-1})+ \lambda(s
            t_j^{-1} s^{-1}) \Bigr)
            \geq c.
        \end{equation*}
Thus $2\re \rho(\sum_{j=1}^n \lambda(t_j)) \ge c>0$, for all $\rho$ in the weak$^*$-closure of  ${\mathrm{conv}}\{s.\omega \mid s \in G\}$, while $\tau_0(\sum_{j=1}^n \lambda(t_j)) = 0$. This shows that (iv) does not hold. 

(v) $\Rightarrow$ (vi). If (v) holds, then for all $t_1,t_2, \dots, t_m \in G \setminus \{e\}$ and all $\ep >0$, there are $s_1,s_2, \dots, s_n \in G$ (repetitions being allowed) such that
$$\Big\|\sum_{k=1}^n \frac{1}{n} \lambda(s_k) \big(\lambda(t_1) + \lambda(t_2) + \cdots + \lambda(t_m)\big) \lambda(s_k)^*\Big\| < \ep.$$
Now use Lemma~\ref{lm:cone} to conclude that (vi) holds. 

(vi) $\Rightarrow$ (i). It is easy to see that (vi) implies the Dixmier property: for each $a \in C^*_\lambda(G)$, $\cconv \{uau^* \mid u \in C^*_\lambda(G) \; \text{unitary}\}$ meets the scalars (necessarily at $\tau_0(a) \! \cdot \! 1$). 
Since $\tau_0$ is faith\-ful, this is easily seen to imply simplicity (and uniqueness of trace) of $C^*_\lambda(G)$, cf.\  \cite{Pow75}. 
\end{proof}

\section{Summary}

\noindent We end with a summary of existing results combined with  results obtained in this article.

\begin{theorem} \label{thm:A}
Let $G$ be a group and let $t \in G$. The following are equivalent:
\begin{enumerate}
\item $t \notin \ra(G)$,
\item there is a boundary action $G \curvearrowright X$ such that $t$ acts non-trivially on $X$,
\item $\tau(\lambda(t)) = 0$, for all tracial states $\tau$ on $C^*_\lambda(G)$,
\item $0 \in \cconv\{ \lambda(sts^{-1}) \mid s \in G \}.$
\end{enumerate}
\end{theorem}

\noindent 
The equivalence of (i) and (ii) is \cite[Proposition 7]{Furman-2003} by Furman (see Theorem~\ref{thm:Furman} above), the equivalence of (ii) and (iii) is \cite[Theorem 4.1]{BKKO} (quoted and reproved here as Theorem~\ref{thm:BKKO}), and the equivalence between (iii) and (iv) is  Theorem~\ref{thm:2} above.

\begin{theorem} \label{thm:B}
Let $G$ be a group. The following are equivalent:
\begin{enumerate}
\item $C^*_\lambda(G)$ has a unique tracial state,
\item $G$ admits a faithful boundary action,
\item $\ra(G) = \{e\}$,
\item for all $t \in G \setminus \{e\}$ and all $\ep>0$, there exist $s_1, s_2, \dots, s_n \in G$ such that 
$$\Big\| \frac{1}{n} \sum_{k=1}^n \lambda(s_kts_k^{-1}) \Big\| < \ep.$$
\end{enumerate}
\end{theorem}

\noindent By universality of the Furstenberg boundary, it follows from Furman's result (cf.\ the equivalence of (i) and (ii) in Theorem~\ref{thm:A} above) that $G$ acts faithfully on its Furstenberg boundary if and only if the amenable radical is trivial. Hence (ii) and (iii) are equivalent.

The equivalence between (i) and (iii) is \cite[Corollary 4.2]{BKKO} (quoted and reproved here as Theorem~\ref{thm:BKKO}), while the equivalence between (iii) and (iv) is Corollary~\ref{cor:unique_trace-2} above.

\begin{theorem} \label{thm:C}
Let $G$ be a group, and let $\tau_0$ be the canonical tracial state on $C^*_\lambda(G)$. The following are equivalent:
\begin{enumerate}
\item $C^*_\lambda(G)$ is simple,
\item $G$ admits a free boundary action,
\item $\tau_0 \in \overline{\{s.\varphi \mid s \in G\}}^{\, w^*}$, for each state $\varphi$ on $C^*_\lambda(G)$,
\item $\tau_0 \in \wscconv\{s.\varphi \mid s \in G\}$, for each state $\varphi$ on $C^*_\lambda(G)$,
\item for all $t_1, t_2, \dots, t_m \in G \setminus \{e\}$ and  all $\ep>0$, there exist $s_1, s_2, \dots, s_n \in G$ such that 
$$\Big\| \frac{1}{n} \sum_{k=1}^n \lambda(s_kt_js_k^{-1}) \Big\| < \ep,$$
for $j=1,2, \dots, m$,
\item $C^*_\lambda(G)$ has the Dixmier property, i.e., 
$\cconv\big\{uau^* \mid u \in C^*_\lambda(G) \; \text{unitary}\big\} \cap \C \! \cdot \! 1 \ne \emptyset$,
for all $a \in C^*_\lambda(G)$. 
\end{enumerate}
\end{theorem}

\noindent The equivalence of (i) and (ii) is stated as Theorem~\ref{thm:KK} above, and was proven by Breuillard--Kalantar--Kennedy--Ozawa. The remaining implications are contained in Theorem~\ref{smpl} (and its proof) above. 

It is worth noting that in (vi) one can even take the unitaries $u$ in $C^*_\lambda(G)$ to be in the set $\{\lambda(t) \mid t \in G\}$. It was shown in \cite{HaaZsido}  that the Dixmier property holds for any unital simple \Cs{} with at most one tracial state. Conversely, any unital \Cs{} satisfying the Dixmier property can have at most one tracial state; moreover, it is simple if, in addition, it has a faithful trace.

It was shown in \cite{BKKO} as a corollary to the characterization of groups with the unique trace property therein, that simplicity of $C^*_\lambda(G)$ implies that $C^*_\lambda(G)$ has unique tracial state. This implication also follows in several different ways from the results obtained in this article. For instance, since $s.\tau = \tau$ for any tracial state $\tau$ on $C^*_\lambda(G)$ and every $s \in G$, the (equivalent) statements (iii) and (iv) in Theorem~\ref{thm:C}  both imply uniqueness of the trace. Respectively, the fact that  Theorem~\ref{thm:C}~(v) clearly implies Theorem~\ref{thm:B}~(iv), yields yet another proof.

Finally, note that the equivalent conditions in Theorem~\ref{thm:C} are strictly stronger than those in Theorem~\ref{thm:A}, due to the very recent results of Le Boudec, \cite{Boudec:C*-simple},  showing that $C^*$-simplicity is not equivalent to the unique trace property.

\noindent 

{\small{
\bibliographystyle{amsplain}
\providecommand{\bysame}{\leavevmode\hbox to3em{\hrulefill}\thinspace}
\providecommand{\MR}{\relax\ifhmode\unskip\space\fi MR }
\providecommand{\MRhref}[2]{%
  \href{http://www.ams.org/mathscinet-getitem?mr=#1}{#2}
}
\providecommand{\href}[2]{#2}

}

\vspace{1cm} \noindent 
Department of Mathematics and Computer Science, University of Southern Denmark, Campusvej 55, 5230 Odense M, Denmark.

\end{document}